\documentclass[reqno, 11pt]{amsart}
\usepackage[top=4cm, bottom=3cm, left=3.2cm, right=3.2cm]{geometry}
\usepackage{graphicx}
\usepackage{float}
\usepackage{enumerate}
\usepackage{amssymb,amsmath}
\usepackage{slashed}
\usepackage{color}
\newtheorem{thm}{Theorem}

\newtheorem{lem}{Lemma}

\newtheorem{defn}{Definition}[section]
\newtheorem{rem}{Remark}
\allowdisplaybreaks

\numberwithin{equation}{section} \numberwithin{lem}{section}
\numberwithin{thm}{section} \numberwithin{prop}{section}
\numberwithin{cor}{section} \numberwithin{rem}{section}
\title[Optimal regularity]{Optimal regularity of subsonic steady-states solution of Euler-Poisson equations for semiconductors with sonic boundary}
\author[S. Li, M. Mei, K. Zhang and G. Zhang]{Siying Li$^{\rm 1}$, Ming Mei$^{\rm 2, 3}$, Kaijun Zhang$^{\rm 1}$ and Guojing Zhang$^{\rm 1, *}$}
\thanks{*Corresponding author.}
\thanks{E-mail addresses: lisiying630@163.com (S. Li), ming.mei@mcgill.ca (M. Mei), zhangkj201@nenu.edu.cn (K. Zhang), zhanggj100@nenu.edu.cn (G. Zhang).}

\begin{document}
  \maketitle
  \begin{center}
	{\footnotesize
		$^{\rm 1}$ School of Mathematics and Statistics, Northeast Normal University,\\
		Changchun, 130024, P.R.China.\\
		\smallskip
		$^{\rm 2}$ Department of Mathematics, Champlain College Saint-Lambert,\\
		Saint-Lambert, Quebec, J4P 3P2, Canada.\\
		$^{\rm 3}$ Department of Mathematics and Statistics, McGill University,\\
		Montreal, Quebec, H3A 2K6, Canada.}
  \end{center}
  \maketitle
  \date{}
	
\begin{abstract}
In this paper, we study the optimal regularity of the stationary sonic-subsonic solution to the unipolar isothermal hydrodynamic model of semiconductors with sonic boundary. Applying the comparison principle and the energy estimate, we obtain the
regularity of the sonic-subsonic solution as $C^{\frac{1}{2}}[0,1]\cap W^{1,p}(0,1)$ for any $p<2$, which is then proved to be optimal by analyzing the property of solution around the singular point on the sonic line, i.e., $\rho\notin C^\nu[0,1]$ for any $\nu>\frac{1}{2}$, and $\rho\notin W^{1,\kappa}(0,1)$ for any $\kappa\ge 2$. Furthermore, we explore the influence of the semiconductors effect on the singularity of solution at sonic points $x=1$ and $x=0$, that is, the solution always has strong
singularity at sonic point $x=1$ for any relaxation time $\tau>0$, but, once the relaxation time is sufficiently large $\tau\gg 1$, then the sonic-subsonic steady-states possess the strong
singularity at both sonic boundaries
$x=0$ and $x=1$. We also show that the pure subsonic
solution $\rho$ belongs to $W^{2,\infty}(0,1)$, which can be embedded into $C^{1,1}[0,1]$, and it is much better than the regularity of sonic-subsonic solutions.
\end{abstract}


\section{Introduction}

This paper is a continuity of the series of previous studies \cite{LMZZ2017, LMZZ2018} on subsonic steady-states for the Euler-Poisson equations with sonic boundary. For the charged fluid particles such as electrons and holes in semiconductor devices, the presented system is the 1-D hydrodynamic model of semiconductors, the so-called Euler-Poisson equations \cite{B1973}:
\begin{equation}
	\left\{
	\begin{aligned}\label{eq1}
		&\rho _t+(\rho u)_x=0,\\
		&(\rho u)_t+\left(\rho u^2+P(\rho)\right)_x=\rho E-\frac{\rho u}{\tau},\\
		&E_x=\rho -b(x),
	\end{aligned}
	\right.
\end{equation}
where $(x,t)\in \Omega\times(0,\infty)$ with $\Omega\subset \mathbb{R}$ is the bounded domain. Without loss of generality, we take $\Omega=[0,1]$. The unknowns $\rho$ and $u$ represent the electron density and the electron velocity, respectively. The function $E$ represents the electric field, which is generated by the Coulomb force of particles. The pressure function $P(\rho)$ denotes the pressure-density relation. When the system is isothermal, $P(\rho)$ is physically represented by
\begin{align}\label{eq2}
P(\rho)=T\rho,\,\,\,\,\,\mbox{with the constant temperature $T>0$}.	
\end{align}
The constant $\tau>0$ denotes the relaxation time. The given function $b(x)>0$ is the doping profile standing for the density of impurities in semiconductor materials. For more details we refer to \cite{J2001, MRS1989}. Throughout the paper, we assume that the doping profile $b(x)\in L^\infty(0,1)$ and denote
\begin{align}\label{eq3}
	\underline{b}:=\underset{x\in(0,1)}{{\rm essinf}}\,\, b(x)\,\,\,\,\,\mbox{and}\,\,\,\,\,\bar{b}:=\underset{x\in(0,1)}{{\rm esssup}}\,\, b(x).	
\end{align}

In this paper, we consider the steady-state equations to \eqref{eq1} in the bounded domain $[0,1]$. Let $J=\rho u$ be the current density of the electrons. Then we have the stationary equations of \eqref{eq1} as follows  
\begin{equation}
	\left\{
	\begin{aligned}\label{eq4}
		&J\equiv constant,\\
		&\left(\frac{J^2}{\rho}+P(\rho)\right)_x=\rho E-\frac{J}{\tau},\,\,\,\,\,\,\, x\in(0,1),\\
		&E_x=\rho-b(x).
	\end{aligned}
	\right.
\end{equation}

According to the terminology from gas dynamics, we call $c:=\sqrt{P^{'}(\rho)}=\sqrt{T}>0$ the sound speed by \eqref{eq2}. Thus, the corresponding electron velocity $u$ of the system \eqref{eq4} is
said to be subsonic / or sonic / or supersonic, if 
\begin{align}\label{eq5}
\mbox{fluid velocity}:\,\, u < / \mbox{or} = / \mbox{or} > c:\,\,\mbox{sound speed}.
\end{align}
Noting that if $(\rho(x),E(x))$ is a pair of solution to equation \eqref{eq4} for a given constant $J$, then $(\rho(1-x),-E(1-x))$ is a solution to equation \eqref{eq4} with respect to $-J$ and $b(1-x)$. So we consider the case of $J>0$. Without loss of generality, we take 
\begin{align}\label{y1}
T=1\,\,\,\,\, \mbox{and} \,\,\,\,\, J=1.
\end{align} 
Therefore, \eqref{eq4} is equivalently reduced to 
\begin{equation}
	\left\{
	\begin{aligned}\label{eq6}
		&\left(\frac{1}{\rho}-\frac{1}{\rho^3}\right)\rho_x=E-\frac{1}{\tau \rho},\,\,\,\,\,\,\, x\in(0,1),\\
		&E_x=\rho-b(x).
	\end{aligned}
	\right.
\end{equation}
In what follows from \eqref{eq5}-\eqref{y1}, it can be identified that $\rho>1$ represents for the subsonic flow, $\rho=1$ for the sonic flow, and $0<\rho<1$ for the supersonic flow. Now we impose \eqref{eq6} with the sonic boundary condition
\begin{align}\label{eq7}
\rho(0)=\rho(1)=1.
\end{align}

Differentiating the equation $\eqref{eq6}_1$ with respect to $x$ and according to the equation $\eqref{eq6}_2$, we derive
\begin{equation}
	\left\{
	\begin{aligned}\label{eq8}
		&\left(\left(\frac{1}{\rho}-\frac{1}{\rho^3}\right)\rho_x\right)_x=\rho-b-\left(\frac{1}{\tau \rho}\right)_x,\\
		&\rho(0)=\rho(1)=1.
	\end{aligned}
	\right.
\end{equation}
When $\rho(x)>1$ for $x\in(0,1)$, the above model \eqref{eq8} is elliptic but denegerate at the boundary, the solution of \eqref{eq8} generally loses certain regularity. So we introduce the
following definition of weak solution by \cite{LMZZ2017}.
\begin{defn}\label{def1}
$\rho(x)$ is called an interior subsonic solution, if $\rho(x)\ge 1$ with $\rho(0)=\rho(1)=1$ to equation \eqref{eq8}, and $\left(\rho(x)-1\right)^2\in H_0^1(0,1)$, and satisfies that, for any $\phi\in H_0^1(0,1)$  
\begin{align*}
\int_{0}^{1}\left[\left(\frac{1}{\rho}-\frac{1}{\rho^3}\right)\rho_x+\frac{1}{\tau \rho}\right]\cdot\phi _xdx+\int_{0}^{1}(\rho-b)\cdot\phi dx=0,		
\end{align*}
namely,
\begin{align}\label{eq9}
\int_{0}^{1}\left[\frac{\rho+1}{2\rho^3}[(\rho-1)^2]_x+\frac{1}{\tau \rho}\right]\cdot \phi_xdx+\int_{0}^{1}\left(\rho-b\right)\cdot\phi dx=0. 	
\end{align}   
\end{defn}

Once $\rho$ is determined from equation \eqref{eq8}, we can further solve the electric field $E(x)$ by 
\begin{align*}
E(x)=\left(\frac{1}{\rho}-\frac{1}{\rho^3}\right)\rho_x+\frac{1}{\tau \rho}=\frac{(\rho+1)[(\rho-1)^2]_x}{2\rho^3}+\frac{1}{\tau \rho}.	
\end{align*}	
Hence, finding the solution of equation \eqref{eq4} with the condition \eqref{eq7} amounts to solving equation \eqref{eq8}.	

The hydrodynamic model for semiconductors, introduced by Bløtekjær \cite{B1973}, has been intensively studied in mathematical physics. In 1990, Degond and Markowich \cite{DM1990} first proved the existence of subsonic solution for one-dimensional case, and showed the uniqueness of solution with small electric current. After that, the steady-state subsonic flows were investigated in various physical boundary conditions and dimensions in \cite{BDC2014, DM1993, GS2005, HMWY2011, LMM2002, NS2007}. For the supersonic steady-state flows, Peng and Violet \cite{PV2006} established the existence and uniqueness of supersonic solution with the semiconductor effect for one-dimensional model. Bae et al \cite{Bae} extended this work to two-dimensional case for pure Euler-Poisson equation. The transonic flows has also been extensively studied in \cite{AMPS1991, DZ2020, G1992, GM1996, LRXX2011, LX2012, R2005}. 

When the boundary  is subjected to be sonic, Li-Mei-Zhang-Zhang \cite{LMZZ2017, LMZZ2018} first investigated in great depth the structure of all types for the doping profile is subsonic and supersonic, respectively. The sonic boundary condition means the system has degeneracy effect, which will make the system has strong singularity. Subsequently, Chen-Mei-Zhang-Zhang \cite{CMZZ2020} further studied the case of transonic doping profile, see also \cite{CMZZ2021, CMZZ2022, CMZZ2023} for the radial
or the spiral radial subsonic, supersonic and transonic solutions in two and three dimensional spaces. Recently, Feng-Hu-Mei \cite{FHM2022} showed the structure stability of different types of solutions. Asymptotic limits of sonic-subsonic solution was studied in \cite{CLMZ2023}. Mu-Mei-Zhang \cite{MMZ2020} proved the well-posedness and ill-posedness of stationary subsonic and supersonic solutions for the bipolar model. However, the literatures mentioned above mostly study the existence and uniqueness of weak solution and other related problems, but there are no more detailed discussions on the regularity of solution for this type of models. The optimal regularity of solution for the stationary semiconductor model with sonic boundary is still
unclear. Thus, exploring this question will be the main goal of this paper.

In this paper, we are devoted to studying the optimal regularity and the corresponding singularity of the sonic-subsonic solution $\rho$ to equation \eqref{eq8} and the corresponding solution $w$ to the following equation transformed from $\rho$  
\begin{equation}
	\left\{
	\begin{aligned}\label{eq10}
		&w_{xx}=\rho-b-\left(\frac{1}{\tau \rho}\right)_x,\,\,\,\,\,\,\, x\in(0,1),\\
		&w(0)=w(1)=\frac{1}{2},
	\end{aligned}
	\right.
\end{equation}
where $w=w(x)={\rm ln}\rho+\frac{1}{2\rho^2}=:F(\rho)$.

Before giving the main results, let us recall the existence and uniqueness of interior subsonic solution, which is excerpted from the first part of Theorem 1.3 in \cite{LMZZ2017}.
\begin{lem}\label{le4}(Existence \cite{LMZZ2017}).
Assume that the doping profile $b(x)\in L^\infty(0,1)$ is subsonic such that $\underline{b}>1$. Then the boundary value problem \eqref{eq6}-\eqref{eq7} admits a unique interior subsonic solution $(\rho,E)\in C^\frac{1}{2}[0,1]\times H^1(0,1)$ satisfying the boundedness
\begin{align}\label{y2}
1+m {\rm sin}\pi x\le \rho(x)\le \bar{b},\,\,\,\,\,\,\, x\in[0,1],
\end{align}
and particularly,
\begin{equation}\label{y3}
\left\{\begin{array}{l}
B_1(1-x)^{\frac{1}{2}} \leq \rho(x)-1 \leq B_2(1-x)^{\frac{1}{2}},\\
-B_3(1-x)^{-\frac{1}{2}} \leq \rho^{\prime}(x) \leq-B_4(1-x)^{-\frac{1}{2}},
\end{array} \quad \text { for } x \text { near } 1,\right.
\end{equation}
where $m=m(\tau,\underline{b})<\bar{b}-1$ is a small positive constant, and $B_2>B_1>0$ and $B_3>B_4>0$ are certain constants. 	
\end{lem}

The main results of this paper are given below.
\begin{thm}\label{th1}
The sonic-subsonic solution $\rho$ obtained in Lemma \ref{le4} admits:

\textbullet \, For $1\le p<2$, the solution $\rho$ to \eqref{eq8} and the solution $w$ to \eqref{eq10} satisfy the following 

\,\,\,\,\,\,\,H\"{o}lder regularity
\begin{equation}
	\left\{
	\begin{aligned}\label{eq11}
		&\rho\in C^{\frac{1}{2}}[0,1],\\
		&w\in C^{1,\frac{1}{2}}[0,1],
	\end{aligned}
	\right.
\end{equation}
\,\,\,\,\,\,\,\,\,\,\,\,\,\,and the following Sobolev regularity
\begin{equation}
	\left\{
	\begin{aligned}\label{eq12}
		&\rho\in W^{1,p}(0,1),\\
		&w\in W^{2,p}(0,1).
	\end{aligned}
	\right.
\end{equation}

\textbullet \, The above regularities are optimal, i.e., $\rho\notin C^\nu[0,1]$ for any $\nu>\frac{1}{2}$, and $\rho\notin W^{1,\kappa}(0,1)$ 

\,\,\,\,\,\,\,\,for any $\kappa\ge 2$.  

\textbullet \, For any relaxation time $\tau$, the solution $\rho$ has the singularity like \eqref{y3} near the point

\,\,\,\,\,\,\,\,$x=1$. The similar singularity exists near the point $x=0$ only for the ralaxation time 

\,\,\,\,\,\,\,\,$\tau\gg 1$.
\end{thm}

\begin{rem}\label{r1}
	
Our first contribution in this article, is to prove the optimal regularities of $\rho$ and $w$. The second contribution is to show the influence of semiconductors effect on the singularity of the solution $\rho$ at sonic points $x=1$ and $x=0$. 

\end{rem}

\begin{rem}\label{r2}
	
It is easy to see that the regularity will be $W^{2,\infty}(0,1)$ and then $C^{1,1}[0,1]$ for the subsonic solution, which is much higher than the case of sonic-subsonic solution.
	
\end{rem}

The remaining part of the present paper is organized as follows. In Section 2, we are devoted to the regularities of $\rho$ and $w$ ultilizing the comparison principle and the energy estimate. In Section 3, by analyzing the properties of solution around the sonic points, we prove the regularities \eqref{eq11}-\eqref{eq12} of solution are optimal. Moreover, we explore the specific influence of the semiconductors effect on the singularity of solution around sonic points $x=1$ and $x=0$.


\section{Regularity of solution}

This section is devoted to studying the regularities of the sonic-subsonic solution $\rho$ to \eqref{eq8} and the solution $w$ to \eqref{eq10}.

{\bf Proof of the first part of Theorem \ref{th1}}.~~First of all, we prove the H\"{o}lder-index of \eqref{eq11} for $\rho$ and $w$. This proof is split into two steps.

Step 1. Preliminary regularity.

Using the weak maximum principle to \eqref{eq8}, we have
\begin{align}\label{eq13}
	1\le \rho \le \bar{b},	
\end{align}
which implies
\begin{align}\label{eq14}
	\frac{1}{2}\le w \le F(\bar{b}).	
\end{align}

Multiplying $\eqref{eq10}_1$ by $w-\frac{1}{2}$ and integrating it by parts over $[0,1]$, and using the Young inequality, we get
\begin{align}\label{eq15}
	\int_{0}^{1}|w_{x}|^2dx&=-\int_{0}^{1}(\rho-b)\cdot\left(w-\frac{1}{2}\right)dx-\frac{1}{\tau}\int_{0}^{1}\frac{1}{\rho}\cdot w_xdx\nonumber\\
	&\le \int_{0}^{1}|\rho-b|\cdot\left|w-\frac{1}{2}\right|dx+\frac{1}{\tau}\int_{0}^{1}\left|\frac{1}{\rho}\right|\cdot|w_x|dx\nonumber\\
	&\le \int_{0}^{1}|\bar{b}|\cdot\left|F(\bar{b})-\frac{1}{2}\right|dx+\frac{1}{\tau}\int_{0}^{1}\left(\frac{\varepsilon}{2}|w_x|^2+\frac{1}{2\varepsilon}\left|\frac{1}{\rho}\right|^2\right)dx,	
\end{align}
where $\varepsilon>0$ is a parameter. Taking $\varepsilon=\tau$, we have
\begin{align}\label{eq16}
	\int_{0}^{1}|w_x|^2dx&\le 2\int_{0}^{1}|\bar{b}|\cdot\left|F(\bar{b})-\frac{1}{2}\right|dx+\frac{1}{\tau^2}\int_{0}^{1}\left|\frac{1}{\rho}\right|^2dx\nonumber\\
	&\le 2\int_{0}^{1}|\bar{b}|\cdot\left|F(\bar{b})-\frac{1}{2}\right|dx+\frac{1}{\tau^2}.	
\end{align}
Therefore, we obtain $w_x\in L^2(0,1)$, namely, $w\in H^1(0,1)$. Since $H^1(0,1)\hookrightarrow C^{\frac{1}{2}}[0,1]$, we get $w\in C^{\frac{1}{2}}[0,1]$.

Step 2. Further regularity.

Since $\rho\in L^2(0,1)$, $b(x)\in L^\infty(0,1)$ and $\left(w_x+\frac{1}{\tau\rho}\right)_x=\rho-b(x)$, we have $\left(w_x+\frac{1}{\tau\rho}\right)_x\in L^\infty(0,1)\subset L^2(0,1)$. Note that $w_x+\frac{1}{\tau\rho}\in L^2(0,1)$, we have $w_x+\frac{1}{\tau\rho}\in H^1(0,1)\hookrightarrow C^\frac{1}{2}[0,1]$. Finally, we derive $w_x\in C^0[0,1]$, which means $w\in C^1[0,1]$.

Since $[(\rho-1)^2]_x=\frac{2\rho^3}{\rho+1}\cdot w_x$, combining $\rho \in C^0[0,1]$ and $\rho+1\neq 0$ for $x\in [0,1]$, we can get $[(\rho-1)^2]_x\in C^0[0,1]$, that is, $(\rho-1)^2\in C^1[0,1]$.

A straightforward calculation gives

\begin{align}\label{eq17}
	\frac{|\rho(x)-\rho(y)|^2}{|x-y|}&=\frac{|(\rho(x)-1)-(\rho(y)-1)| \cdot |(\rho(x)-1)-(\rho(y)-1)|}{|x-y|}\nonumber\\
	&\le \frac{|(\rho(x)-1)-(\rho(y)-1)| \cdot |(\rho(x)-1)+(\rho(y)-1)|}{|x-y|}\nonumber\\
	&=\frac{|[\rho(x)-1]^2-[\rho(y)-1]^2|}{|x-y|}\nonumber\\
	&< \infty,
\end{align}
for any $x, y\in [0,1]$ and $x\ne y$. Thus, we derive $\rho \in C^\frac{1}{2}[0,1]$, which means $w_x \in C^\frac{1}{2}[0,1]$. This, together with $w\in C^1[0,1]$, indicates $w\in C^{1,\frac{1}{2}}[0,1]$. 
 
Next, we prove the Sobolev-index of \eqref{eq12} for $\rho$ and $w$. The proof is divided into three steps.

Step 1. We prove that for any $\alpha>0$, it holds $(\rho-1)^\alpha \rho_x\in L^2(0,1)$.

Multiplying $\eqref{eq8}_1$ by $((\rho-1)^{2\alpha}-\mu)^+={\rm max}\left\{(\rho-1)^{2\alpha}-\mu,0\right\}$ and integrating the resulted equation over $[0,1]$, after defining the set $\Omega_\mu:=\left\{x| (\rho-1)^{2\alpha}\ge\mu\right\}$, where $\mu>0$, we have
\begin{align}\label{eq18}
	&-\int_{\Omega_\mu}\left(\frac{\rho+1}{\rho^3}(\rho-1)\rho_x\right)\cdot\left((\rho-1)^{2\alpha}-\mu\right)_x^+dx-\int_{\Omega_\mu}(\rho-1)\cdot((\rho-1)^{2\alpha}-\mu)^+dx\nonumber\\
	=&\int_{\Omega_\mu}(1-b)\cdot((\rho-1)^{2\alpha}-\mu)^+dx-\int_{\Omega_\mu}\frac{1}{\tau}\left(\frac{1}{\rho}\right)_x\cdot ((\rho-1)^{2\alpha}-\mu)^+dx.
\end{align}
For the sake of simplicity, we mark the above equation as $L_1+L_2=R_1+R_2$. By some
straightforward computations, it follows that
\begin{align}\label{eq19}
	L_1&=-2\alpha \int_{\Omega_\mu}\left(\frac{\rho+1}{\rho^3}(\rho-1)\rho_x\cdot(\rho-1)^{2\alpha-1}\rho_x\right)dx\nonumber\\
	&=-2\alpha \int_{\Omega_\mu}\frac{\rho+1}{\rho^3}(\rho-1)^{2\alpha }|\rho_x|^2dx,
\end{align}
\begin{align}\label{eq20}
	L_2\le 0,
\end{align}
\begin{align}\label{eq21}
	|R_1|=\left|\int_{\Omega_\mu}(1-b)\cdot((\rho-1)^{2\alpha}-\mu)^+\right|<\bar{b}\cdot(\bar{b}-1)^{2\alpha}\cdot|\Omega_\mu|<\infty	
\end{align}
and
\begin{align}\label{eq22}
	R_2&=2\alpha \int_{\Omega_\mu}\frac{1}{\tau}\left(\frac{1}{\rho}(\rho-1)^{2\alpha-1}\rho _x\right)dx\nonumber\\
	&=2\alpha \int_{\Omega_\mu}\frac{1}{\tau}\cdot\left(\frac{1}{1+(\rho-1)}(\rho-1)^{2\alpha-1}(\rho-1)_x\right)dx\nonumber\\
	&=\frac{2\alpha}{\tau}\int_{\Omega_\mu}\left(G(\rho-1)\right)_xdx\nonumber\\
	&=\frac{2\alpha}{\tau}\cdot G(\rho-1)|_{\partial\Omega_\mu}=0,
\end{align}
where $G'(s):=\frac{s^{2\alpha-1}}{1+s}$. Substituting \eqref{eq19}-\eqref{eq22} into \eqref{eq18}, we have
\begin{align}\label{eq23}
	2\alpha \int_{\Omega_\mu}\frac{\rho+1}{\rho^3}(\rho-1)^{2\alpha }|\rho_x|^2dx<\infty.	
\end{align}
According to $C\int_{\Omega_\mu}(\rho-1)^{2\alpha}|\rho_x|^2dx\le\int_{\Omega_\mu}\frac{\rho+1}{\rho^3}(\rho-1)^{2\alpha }|\rho_x|^2dx$, and together with \eqref{eq23}, we have
\begin{align}\label{eq25}
	\int_{\Omega_\mu}(\rho-1)^{2\alpha}|\rho_x|^2dx\le C,	
\end{align}
here and below the constant $C>0$ is independent on $\mu$. Let $\mu\to 0$, we get
\begin{align}\label{eq26}
	\int_{0}^{1}(\rho-1)^{2\alpha}|\rho_x|^2dx\le C,	
\end{align}
which indicates
\begin{align}\label{eq27}
	(\rho-1)^\alpha\rho_x\in L^2(0,1).	
\end{align}

Step 2. We prove that $\rho(x)-1 \ge \beta {\rm sin}\pi x$ holds for $\beta>0$ sufficiently small.

Taking $u=\frac{1}{2}+\gamma {\rm sin}^2\pi x$, where $\gamma>0$ is a constant, we naturally have $u_x=2\gamma \pi {\rm sin}\pi x{\rm cos}\pi x$, $u_{xx}=2\gamma\pi^2{\rm cos}2\pi x$. Now define $f(y):=F^{-1}(y)$, where $y\ge \frac{1}{2}$, then we derive  
\begin{equation}
	\left\{
	\begin{aligned}\label{eq28}
		&u_{xx}-(f(u)-1)+\left(\frac{1}{\tau f(u)}\right)_x=2\gamma\pi^2{\rm cos}2\pi x-(f(u)-1)+\left(\frac{1}{\tau f(u)}\right)_x,\\
		&u(0)=u(1)=\frac{1}{2}.
	\end{aligned}
	\right.
\end{equation}
For equation \eqref{eq28}, it is easy to see that
\begin{align}\label{eq29}
	\left\|2\gamma\pi^2{\rm cos}2\pi x\right\|_{L^\infty(0,1)}\to 0,\,\,\, \mbox{as}\,\,\, \gamma \to 0.	
\end{align}
When $\gamma \to 0$, we have $u(x)=\frac{1}{2}+\gamma {\rm sin}^2\pi x\to\frac{1}{2}$, and since $F(1)=\frac{1}{2}$, we infer that $f(u)\to 1$. So we get
\begin{align}\label{eq30}
	\left\|f(u)-1\right\|_{L^\infty(0,1)}\to 0,\,\,\,\mbox{as}\,\,\,\gamma \to 0.	
\end{align}
Because $1\le s \le \bar{b}$, we have
\begin{align*}
	F'(s)=\frac{1}{s}-\frac{1}{s^3}=\frac{(s+1)(s-1)}{s^3}\ge C(s-1).
\end{align*}
Thus, it follows that
\begin{align}\label{eq31}
	f'(u)=(F^{-1}(u))'=\frac{1}{F'(F^{-1}(u))}\le \frac{1}{C(F^{-1}(u)-1)}.
\end{align}
According to the Taylor expansion around $\rho=1$, we derive 
\begin{align}\label{eq32}
	F(\rho)&=\frac{1}{2}+F'(1)(\rho-1)+\frac{F''(1)}{2}(\rho-1)^2+o(1)(\rho-1)^2\nonumber\\
	&=\frac{1}{2}+(\rho-1)^2+o(1)(\rho-1)^2.
\end{align}
It is easy to obtain that for $\rho$ close to $1$
\begin{align}\label{eq33}
	(\rho-1)^2\ge C\left(F(\rho)-\frac{1}{2}\right).	
\end{align}
For $\rho$ not close to $1$, we also have 
\begin{align}\label{eq34}
	(\rho-1)^2\ge C\left(F(\rho)-\frac{1}{2}\right).
\end{align} 
Combining \eqref{eq33}-\eqref{eq34}, it holds that $\left(\rho-1\right)^2\ge C\left(F\left(\rho\right)-\frac{1}{2}\right)$ for any $x\in[0,1]$. Then we obtain
\begin{align}\label{eq35}
	F^{-1}(u)-1\ge C\sqrt{u-1}.	
\end{align}
This, together with \eqref{eq31} and \eqref{eq35}, indicates
\begin{align}\label{eq36}
	f'(u)\le \frac{C}{\sqrt{u-\frac{1}{2}}}.
\end{align}
Therefore, we get
\begin{align}\label{eq37}
	\left|\left(\frac{1}{\tau f(u)}\right)_x\right|&\le\frac{1}{\tau}\frac{\left|f'(u)\right|\left|u_x\right|}{f^2(u)}\nonumber\\
	&\le \frac{1}{\tau}\cdot\frac{1}{f^2(u)}\left(\frac{C}{\sqrt{u-\frac{1}{2}}}\right)\cdot u_x\nonumber\\
	&\le \frac{1}{\tau}\frac{C}{\sqrt{\gamma{\rm sin}^2\pi x}}\cdot 2\gamma \pi {\rm sin}\pi x{\rm cos}\pi x\nonumber\\
	&=\frac{C\cdot 2\pi}{\tau}\cdot \sqrt{\gamma}{\rm cos}\pi x. 
\end{align}
Then we have
\begin{align}\label{eq38}
	\left\|\left(\frac{1}{\tau f(u)}\right)_x\right\|_{L^\infty(0,1)}\to 0,\,\,\, \mbox{as}\,\,\, \gamma\to 0.	
\end{align}
Substituting \eqref{eq29}-\eqref{eq30} and \eqref{eq38} into \eqref{eq28}, we arrive 
\begin{equation}
	\left\{
	\begin{aligned}\label{eq39}
		&u_{xx}-(f(u)-1)+\left(\frac{1}{\tau f(u)}\right)_x=g(\gamma,x),\\
		&u(0)=u(1)=\frac{1}{2},
	\end{aligned}
	\right.
\end{equation}
where $g(\gamma,x)$ represents the right side term of equation $\eqref{eq28}_1$. From the above analysis, we know that $\left\|g(\gamma,x)\right\|_{L^\infty(0,1)}\to 0$ as $\gamma \to 0$. 

From \eqref{eq10}, $w(x)$ satisfies the following equation
\begin{equation}
	\left\{
	\begin{aligned}\label{eq40}
		&w_{xx}-(f(w)-1)+\left(\frac{1}{\tau f(w)}\right)_x=1-b<0,\\
		&w(0)=w(1)=\frac{1}{2}.
	\end{aligned}
	\right.
\end{equation}

By taking the difference between \eqref{eq39} and \eqref{eq40}, and setting $V=u-w$, it follows that
\begin{equation}
	\left\{
	\begin{aligned}\label{eq41}
		&V_{xx}-(f(u)-f(w))+\left(\frac{1}{\tau f(u)}-\frac{1}{\tau f(w)}\right)_x=g(\gamma,x)-1+b,\\
		&V(0)=V(1)=0.
	\end{aligned}
	\right.
\end{equation}
For the above equation, there must exist $\gamma_0$ small enough such that $g(\gamma_0,x)-1+b\ge 0$. Now we prove that $V\le 0$ for any $x\in [0,1]$. Taking $\varphi=\frac{V^+}{V^++h}$ as the test function according to the comparison principle in \cite[Theorem 10.7]{GT2001}, where $h>0$, then we get $\varphi_x=\frac{h V^+_x}{(V^++h)^2}$ and $\frac{V^+_x}{V^++h}=\left({\rm log}\left(1+\frac{V^+}{h}\right)\right)_x$. Multiplying $\eqref{eq41}_1$ by $\varphi$ and integrating it by parts over $[0,1]$, we derive
\begin{align}\label{eq42}
	&h\int_{0}^{1}\left|\left({\rm log}\left(1+\frac{V^+}{h}\right)\right)_x\right|^2dx+\int_{0}^{1}(f(u)-f(w))\cdot \frac{V^+}{V^++h}dx\nonumber\\
	=&\frac{1}{\tau}\int_{0}^{1}\frac{f(u)-f(w)}{f(u)f(w)}\cdot \frac{h V^+_x}{(V^++h)^2}dx-\int_{0}^{1}(g-1+b)\cdot \frac{V^+}{V^++h}dx. 	
\end{align}
Obviously, $\int_{0}^{1}(g-1+b)\cdot \frac{V^+}{V^++h}dx\ge 0$, and it follows that 
\begin{align}\label{eq43}
	&h\int_{0}^{1}\left|\left({\rm log}\left(1+\frac{V^+}{h}\right)\right)_x\right|^2dx+\int_{0}^{1}(f(u)-f(w))\cdot \frac{V^+}{V^++h}dx\nonumber\\
	\le&\frac{1}{\tau}\int_{0}^{1}\frac{f(u)-f(w)}{f(u)f(w)}\cdot \frac{h V^+_x}{(V^++h)^2}dx\nonumber\\
	\le &\frac{h}{\tau}\int_{0}^{1}\frac{1}{f(u)f(w)}\cdot\left|\frac{f(u)-f(w)}{V^++h}\right|\cdot\left|\frac{V^+_x}{V^++h}\right|dx=:I.
\end{align}
Because $f(u)=f(w)=1$ at $x=0$ and $x=1$, it leads to $\underset{h\to0^+}{{\rm lim}} \left|\frac{f(u)-f(w)}{V^++h}\right|\to+\infty$ around $x=0$ or $x=1$. For this reason, the comparison principle in \cite[Theorem 10.7]{GT2001} can not be directly applied. Now define the set $T:=\left\{x\in[0,1]|\left|\frac{f(u)-f(w)}{V^++h}\right|\le \mathcal{C}\right\}$, where $\mathcal{C}>0$ is to be determined. Applying the Young inequality, we derive
\begin{align}\label{eq44}
	I\le& \frac{h}{\tau}\int_{T}\mathcal{C}\left|\left({\rm log}\left(1+\frac{V^+}{h}\right)\right)_x\right|dx+\frac{h}{\tau}\int_{T^c}\left|\frac{f(u)-f(w)}{V^++h}\right|\cdot\left|\left({\rm log}\left(1+\frac{V^+}{h}\right)\right)_x\right|dx\nonumber\\
	\le&\frac{h\mathcal{C}}{\tau}\int_{0}^{1}\left|\left({\rm log}\left(1+\frac{V^+}{h}\right)\right)_x\right|dx\nonumber\\
	&+\frac{h}{\tau}\int_{T^c}\left[\frac{1}{2\varepsilon}\left|\frac{f(u)-f(w)}{V^++h}\right|^2+\frac{\varepsilon}{2}\left|\left({\rm log}\left(1+\frac{V^+}{h}\right)\right)_x\right|^2\right]dx,	
\end{align}
where $\varepsilon>0$ denotes the parameter, and $T^c=[0,1]/T$. Taking $\varepsilon=\tau$, we get
\begin{align}\label{eq45}
	I&\le\frac{h\mathcal{C}}{\tau}\int_{0}^{1}\left|\left({\rm log}\left(1+\frac{V^+}{h}\right)\right)_x\right|dx+\frac{h}{2}\int_{0}^{1}\left|\left({\rm log}\left(1+\frac{V^+}{h}\right)\right)_x\right|^2dx\nonumber\\
	&\,\,\,\,\,\,\,\,+\frac{1}{2\tau^2}\int_{T^c}h\left|\frac{f(u)-f(w)}{V^++h}\right|^2dx.
\end{align}
By recalling the Taylor expansion in \eqref{eq32} and using the properties of the functions $F$ and $f$, we obtain
\begin{align}\label{eq46}
	h\left|\frac{f(u)-f(w)}{V^++h}\right|^2\le\frac{\left|f(u)-f(w)\right|^2}{V^++h}\le\frac{\left|f(u)-f(w)\right|^2}{V^+}<M,	
\end{align} 
where $M>0$ is independent on $h$. Plugging \eqref{eq45}-\eqref{eq46} into \eqref{eq43}, we have
\begin{align}\label{eq47}
	&h\int_{0}^{1}\left|\left({\rm log}\left(1+\frac{V^+}{h}\right)\right)_x\right|^2dx+2\int_{0}^{1}(f(u)-f(w))\cdot \frac{V^+}{V^++h}dx\nonumber\\
	\le&\frac{2h\mathcal{C}}{\tau}\int_{0}^{1}\left|\left({\rm log}\left(1+\frac{V^+}{h}\right)\right)_x\right|dx+\frac{1}{\tau^2}\int_{T^c}Mdx.
\end{align}

Next we analyze $\int_{0}^{1}(f(u)-f(w))\cdot \frac{V^+}{V^++h}dx$. Suppose that $V^+\not\equiv 0$, based on the monotonicity increasing properties of $F$ and $f$, we have
\begin{align}\label{eq48}
	\underset{h\to0^+}{{\rm lim}}\int_{0}^{1}(f(u)-f(w))\cdot \frac{V^+}{V^++h}dx>C>0.
\end{align}
There exists a constant $\mathcal{C}\gg 1$ such that the measure $|T^c|\ll 1$ satisfying
\begin{align}\label{eq49}
	\frac{1}{\tau^2}\int_{T^c}Mdx\le \int_{0}^{1}(f(u)-f(w))\cdot \frac{V^+}{V^++h}dx. 	
\end{align}
Combining \eqref{eq47}-\eqref{eq49}, we obtain
\begin{align}\label{eq50}
	\int_{0}^{1}\left|\left({\rm log}\left(1+\frac{V^+}{h}\right)\right)_x\right|^2dx\le\frac{2\mathcal{C}}{\tau}\int_{0}^{1}\left|\left({\rm log}\left(1+\frac{V^+}{h}\right)\right)_x\right|dx.	
\end{align}
Let $C=\frac{2\mathcal{C}}{\tau}$, and by the H\"{o}lder inequality, we have
\begin{align}\label{eq51}
	\left\|\left({\rm log}\left(1+\frac{V^+}{h}\right)\right)_x\right\|_{L^2(0,1)}^2\le C \left\|\left({\rm log}\left(1+\frac{V^+}{h}\right)\right)_x\right\|_{L^2(0,1)},\,\,\, \mbox{as} \,\,\, h\to 0^+,	
\end{align}
which gives
\begin{align}\label{eq52}
	\left\|\left({\rm log}\left(1+\frac{V^+}{h}\right)\right)_x\right\|_{L^2(0,1)}\le C,\,\,\, \mbox{as} \,\,\, h\to 0^+.	
\end{align}
Using the Poincar\'e inequality, we have
\begin{align}\label{eq53}
	\left\|\left({\rm log}\left(1+\frac{V^+}{h}\right)\right)\right\|_{L^2(0,1)}\le C,\,\,\, \mbox{as} \,\,\, h\to 0^+.	
\end{align}
So we get $V^+\equiv 0$, which means $u\le w$.

Step 3. We prove that $\rho_x\in L^p(0,1)$ for any $p<2$.

We know that $\rho\ge 1+\beta{\rm sin}\pi x$ by step 2, i.e., $\rho-1>0$ for any $x\in (0,1)$, then we have
\begin{align}\label{eq54}
	|\rho_x|=\left|\frac{(\rho-1)^\alpha\rho_x}{(\rho-1)^\alpha}\right|=\frac{|(\rho-1)^\alpha\rho_x|}{|\rho-1|^\alpha}.
\end{align}
By step 1, we know $(\rho-1)^\alpha \rho_x\in L^2(0,1)$. And according to the conclusion of step 2, we naturally get $\frac{1}{(\rho-1)^\alpha}\le \frac{1}{C({\rm sin}\pi x)^\alpha}$. Therefore, for any $q<\infty$, we can take $\alpha\ll 1$ such that $\frac{1}{(\rho-1)^\alpha}\in L^q(0,1)$. By the H\"{o}lder inequality, we obtain $\rho_x\in L^p(0,1)$ for any $p<2$, and thus $\rho \in W^{1,p}(0,1)$ for any $p<2$.

By $\eqref{eq2}_1$, we derive $w \in W^{2,p}(0,1)$ for any $p<2$. \hfill $\Box$ 


\section{Optimal regularity}

In this section, we prove that the regularities of $\rho$ and $w$ proved in Section 2 are optimal. Firstly, we prove that $\rho$ has singularity at sonic point $x=1$. This casues the regularity of $\rho$ is optimal, that is, $\rho\notin C^\nu[0,1]$ for any $\nu>\frac{1}{2}$, and $\rho\notin W^{1,\kappa}(0,1)$ for any $\kappa\ge 2$.
\begin{lem}\label{le1}
For any relaxation time $\tau>0$, the first derivative of the solution $w$ to equation \eqref{eq10} satisfies $w_x(1)<0$.
\end{lem}
\begin{proof}
Since $w\in C^{1,\frac{1}{2}}(0,1)$, $w\ge \frac{1}{2}$ for $x\in[0,1]$ and $w(0)=w(1)=\frac{1}{2}$, we know $w_x(1)\le 0$. We prove this lemma by contradiction. Assume that $w_x(1)=0$. Due to $b>1$, we can suppose $b\ge 1+2c$, where $c>0$. Moreover, according to $\rho\in C^\frac{1}{2}(0,1)$ and $\rho(1)=1$, there exists $0<\delta<1$ such that $\rho\le 1+c$ for $x\in [1-\delta,1]$. This means $\rho-b\le-c$ for $x\in[1-\delta,1]$. By the Newton-Leibniz formula, we have 
\begin{align}\label{eq55}
w_x(x)&=w_x(1)-\int_{x}^{1}w_{xx}(s)ds\nonumber\\
&=-\int_{x}^{1}w_{xx}(s)ds\nonumber\\
&=-\int_{x}^{1}\left(\rho-b-\left(\frac{1}{\tau \rho}\right)_x\right)(s)ds\nonumber\\
&>c(1-x)+\frac{1}{\tau}-\frac{1}{\tau \rho(x)}\nonumber\\
&>0
\end{align}
for any $x\in[1-\delta,1]$. Thus, $w$ is strictly monotonically increasing over $[1-\delta,1]$, which contradicts to the fact that $w(x)\ge \frac{1}{2}$ over $[0,1]$. Then, it holds that $w_x(1)<0$. Therefore, we complete the proof of Lemma \ref{le1}.	
\end{proof}
Now, we are ready to prove the rest parts of Theorem \ref{th1}.

{\bf Proof of the second and third parts of Theorem \ref{th1}}.~~The function $w_x$ is continuous due to $w\in C^{1,\frac{1}{2}}(0,1)$. Because $w_x(1)<0$ by Lemma \ref{le1}, there must exist three constants $\delta_1$, $C_1$, $C_2$ with $0<\delta_1<1$ and $C_2>C_1>0$, such that $-C_2\le w_x(x)\le -C_1$ over $[1-\delta_1,1]$. Thus 
\begin{align}\label{eq56}
	C_1(1-x)\le w-\frac{1}{2}\le C_2(1-x)
\end{align}
for any $x\in [1-\delta_1,1]$. This indicates the function $w-\frac{1}{2}$ is between the functions $y=C_1(1-x)$ and $y=C_2(1-x)$ for any $x\in[1-\delta_1,1]$.

By \eqref{eq32} and \eqref{eq56}, we obtain
\begin{align}\label{eq57}
	k_1(\rho-1)^2\le w(\rho)-\frac{1}{2}\le k_2(\rho-1)^2,
\end{align} 
where $k_1, k_2$ denote two positive constants. This, together with  \eqref{eq56}-\eqref{eq57}, gives
\begin{align}\label{eq58}
	h_1(1-x)^\frac{1}{2}\triangleq\sqrt{\frac{C_2}{k_2}}(1-x)^\frac{1}{2}\le \rho-1\le \sqrt{\frac{C_1}{k_1}}(1-x)^\frac{1}{2}\triangleq h_2(1-x)^\frac{1}{2}
\end{align}
for any $x\in[1-\delta_1,1]$. This implies the function $\rho-1$ lies between the functions $y=h_1(1-x)^\frac{1}{2}$ and $y=h_2(1-x)^\frac{1}{2}$ for any $x\in[1-\delta_1,1]$.

Thus, for any $\nu>\frac{1}{2}$, we have
\begin{align}\label{eq59}
	\left|\frac{\rho(1)-\rho(x)}{(1-x)^\nu}\right|\ge\frac{h_1(1-x)^\frac{1}{2}}{(1-x)^\nu}=\frac{h_1}{(1-x)^{\nu-\frac{1}{2}}}\to\infty,\,\,\, \mbox{as}\,\,\, x\to 1.	
\end{align} 
This shows $\rho \notin C^\nu[0,1]$ for any $\nu>\frac{1}{2}$.

Moreover, according to $\rho_x=\frac{\rho^3}{\rho+1}\cdot\frac{w_x}{\rho-1}$ and $-C_2\le w_x(x)\le -C_1$ over $[1-\delta_1,1]$, we get
\begin{align}\label{eq60}
	\left|\rho_x\right|=\left|\frac{\rho^3}{\rho+1}\right|\cdot\left|\frac{w_x}{\rho-1}\right|\ge\frac{1}{\bar{b}+1}\cdot\frac{C_1}{\rho-1}.
\end{align}
Substituing \eqref{eq58} into \eqref{eq60}, we derive
\begin{align}\label{eq61}
	\left|\rho_x\right|\ge\frac{C_1}{h_2(\bar{b}+1)}\cdot\frac{1}{(1-x)^\frac{1}{2}}\notin L^2[1-\delta_1,1],	
\end{align}
and thus $\rho\notin W^{1,\kappa}(0,1)$ for any $\kappa\ge2$. \hfill $\Box$

The above results indicate that for any relaxation time $\tau>0$, the sonic-subsonic solution $\rho$ has strong singularity at sonic point $x=1$. This leads to that the regularities of $\rho$ and $w$ proved in Section 2 are optimal. 

Next, we focus on the singularity of $\rho$ at sonic point $x=0$. In 2017, Li-Mei-Zhang-Zhang showed that there is no singularity when the relaxation time $\tau\ll 1$, see \cite[Theorem 2.4]{LMZZ2017}. However, we can prove that there is still strong singularity to $\rho$ at sonic point $x=0$ when the relaxation time $\tau\gg 1$. To this end, we establish the following lemma.
\begin{lem}\label{le2}
For the relaxation time $\tau\gg 1$, the first derivative of the solution $w$ to equation \eqref{eq10} satisfies $w_x(0)>0$.
\end{lem}
\begin{proof}
Since $w\in C^{1,\frac{1}{2}}(0,1)$, $w\ge\frac{1}{2}$ for $x\in[0,1]$ and $w(0)=w(1)=\frac{1}{2}$, we know $w_x(0)\ge 0$. We prove this lemma also by contradiction. Suppose $w_x(0)=0$. Since $b>1$, we can assume $b\ge 1+2c$, where $c>0$. Furthermore, because $\rho \in C^\frac{1}{2}(0,1)$ and $\rho(0)=1$, there is $0<\eta<1$ such that $\rho\le 1+c$ for $x\in[0,\xi]$. This implies $\rho-b\le-c$ for $x\in[0,\xi]$. Using the Newton-Leibniz formula, for any $x\in[0,\xi]$, we derive
\begin{align}\label{eq62}
w_x(x)&=w_x(0)+\int_{0}^{x}w_{xx}(s)ds\nonumber\\
&=\int_{0}^{x}w_{xx}(s)ds\nonumber\\
&=\int_{0}^{x}\left(\rho-b-\left(\frac{1}{\tau \rho}\right)_x\right)(s)ds\nonumber\\
&\le-cx-\frac{1}{\tau \rho(x)}+\frac{1}{\tau}\nonumber\\
&\le \frac{1}{\tau}.
\end{align}

Applying the Newton-Leibniz formula again for any $x\in[0,\xi]$, we have 
\begin{align}\label{eq63}
w(x)\le w(0)+\int_{0}^{x}\frac{1}{\tau}ds\le \frac{1}{2}+\frac{\eta}{\tau}\le \frac{1}{2}+\frac{1}{\tau}.	
\end{align}
In view of \eqref{eq63}, taking $\tau_0>0$ such that $\frac{1}{2}+\frac{1}{\tau_0}=F\left(1+\frac{c}{2}\right)$. When $\tau\ge\tau_0$, we have for any $x\in[0,\xi]$
\begin{align}\label{eq64}
F(\rho)=w(x)\le \frac{1}{2}+\frac{1}{\tau}\le \frac{1}{2}+\frac{1}{\tau_0}=F\left(1+\frac{c}{2}\right).
\end{align}
So we have $\rho\le1+\frac{c}{2}$ for any $x\in[0,\xi]$. 

By continuous extension method, we can ultimately obtain that $w\le \frac{1}{2}+\frac{1}{\tau}$ and $\rho<1+c$ for all $x\in[0,1]$. Obviously, $w\left(\frac{1}{2}\right)\le \frac{1}{2}+\frac{1}{\tau}$. For any $y\in[\frac{1}{2},1]$, by the Newton-Leibniz formula, we have
\begin{align}\label{eq65}
w_x(y)&=w_x(0)+\int_{0}^{y}w_{xx}(s)ds\nonumber\\
&=\int_{0}^{y}\left(\rho-b-\left(\frac{1}{\tau \rho}\right)_x\right)(s)ds\nonumber\\
&<-\frac{c}{2}+\frac{1}{\tau}.	
\end{align}
Taking $\tau\ge{\rm max}\left\{\tau_0,\frac{4}{c}\right\}$, we derive $w_x(y)<-\frac{c}{4}$, and thus
\begin{align}\label{eq66}
w(1)&=w(\frac{1}{2})+\int_{\frac{1}{2}}^{1}w_{x}(s)ds\nonumber\\
&<\frac{1}{2}+\frac{1}{\tau}+\int_{\frac{1}{2}}^{1}-\frac{c}{4}ds\nonumber\\
&<\frac{1}{2}+\frac{1}{\tau}-\frac{c}{8}.	
\end{align}
Let us take $\tau\ge{\rm max}\left\{\tau_0,\frac{8}{c}\right\}$ again, which together with \eqref{eq66} implies $w(1)<\frac{1}{2}$. This is contradictory to the fact that $w\ge \frac{1}{2}$ over $[0,1]$. Then we have $w_x(0)>0$. The proof of Lemma \ref{le2} is completed. 
\end{proof}

Next, we repeat the process of proof at sonic point $x=1$, and ultimately obtain that the sonic-subsonic solution $\rho$ still exists strong singularity at sonic point $x=0$ when the relaxation time $\tau\gg 1$.


\section*{Acknowledgements}
The research of M.~Mei was supported by NSERC grant RGPIN 2022--03374, the research of K.~Zhang was supported by National Natural Science Foundation of China grant 12271087, and the research of G.~Zhang was supported by National Natural Science Foundation of China grant 11871012.


\appendix

\end{document}